\documentclass[reqno]{amsart}
\usepackage{float}
\usepackage{amsmath}
\usepackage{graphicx}
\usepackage{latexsym}
\usepackage{amsfonts}
\usepackage{amssymb}
\theoremstyle{plain}
\newtheorem*{theorem}{Theorem}
\newtheorem*{lemma}{Lemma}
\theoremstyle{definition}

\begin{document}
\title[Heavy-traffic analysis]{Heavy-traffic analysis of the maximum of an asymptotically stable random walk}
\author[Shneer]{Seva Shneer}
\address{Department of Mathematics and Computer Science, Eindhoven University of Technology, P.O.Box 513, 5600 MB Eindhoven, The Netherlands}
\email{shneer@eurandom.tue.nl}

\author[Wachtel]{Vitali Wachtel}
\address{Mathematical Institute, University of Munich, Theresienstrasse 39, D--80333
Munich, Germany}
\email{wachtel@mathematik.uni-muenchen.de}
\begin{abstract}
For families of random walks $\{S_k^{(a)}\}$ with $\mathbf E S_k^{(a)} = -ka < 0$ we consider their maxima $M^{(a)} = \sup_{k \ge 0} S_k^{(a)}$. We investigate the asymptotic behaviour of $M^{(a)}$ as $a \to 0$ for asymptotically stable random walks. This problem appeared first in the 1960's in the analysis of a single-server queue when the traffic load tends to $1$ and since then is referred to as the heavy-traffic approximation problem. Kingman and Prokhorov suggested two different approaches which were later followed by many authors. We give two elementary proofs of our main result, using each of these approaches. It turns out that the main technical difficulties in both proofs are rather similar and may be resolved via a generalisation of the Kolmogorov inequality to the case of an infinite variance. Such a generalisation is also obtained in this note.
\end{abstract}

\maketitle

Assume that $\{X_i\}_{i=1}^\infty$ is a sequence of i.i.d. random
variables with a zero expectation: $\mathbf EX_1=0$.
Define a random walk
$$
S_0 = 0, \quad S_k = \sum_{i=1}^k X_i \quad \text{for} \quad k \ge
1.
$$
Along with the random walk $\{S_k\}$, for each $a>0$ define a random walk $\{S_k^{(a)}\}$ via
$$
S_k^{(a)} = S_k - ka.
$$
Now we can define
$$
M^{(a)} = \sup_{k \ge 0} S_k^{(a)}.
$$
Since the random walk $S^{(a)}_k$ drifts to $-\infty$,
the maximum $M^{(a)}$ is a proper random variable for each $a>0$.
However, $M^{(a)} \to +\infty$ in probability as $a\to 0$. From this fact
a natural question arises: How fast does $M^{(a)}$ grow as $a\to0$?

It is well-known that the stationary distribution of the waiting
time of a customer in a single-server first-come-first-served
($GI/GI/1$) queue coincides with that of the maximum of a
corresponding random walk. The condition on the mean of the random
walk becoming small ($a\to0$) means in the context of a queue that
the traffic load tends to $1$. Thus, the problem under
consideration may be seen as the investigation of the growth rate
of the stationary waiting-time distribution in a $GI/GI/1$ queue.
This is one of the most important problems in the queueing theory
that is referred to as the heavy-traffic analysis. The question
was first posed by Kingman \cite{Kin61,Kin62} (see also
\cite{Kin65} for an extensive discussion) who considered the case
when $|X_1|$ has an exponential moment and proved that
$$
\mathbf P(aM^{(a)} \leq x) \to 1 - e^{-2x/{\sigma^2}}
$$
for all $x \ge 0$ as $a \to 0$, where $\sigma^2 = Var X_1$. In his proof, Kingman used an analytical approach. Namely, he applied an expression for the Laplace transform of $M^{(a)}$ which follows from the Wiener-Hopf factorisation.

Prokhorov \cite{Prokh63} generalised the latter result to the case when $Var X_1 \in (0,\infty)$. Prokhorov's approach was based on the functional Central Limit Theorem.

These two approaches have become classical and have both been used to prove various heavy-traffic results. However, they seem to have never been compared. Later in this note we shall discuss both approaches in more detail and point out their differences and similarities. The analytical approach was used by Boxma and Cohen \cite{BC99} (see also Cohen \cite{Cohen02}) to study the limiting behaviour of $M^{(a)}$ in the case of an infinite variance. They proved that if $\mathbf P(X_1>x)$ is regularly varying at infinity with a parameter $1<\alpha<2$ (and under some additional assumptions), then there exists a function $\Delta(a)$ such that $\Delta(a) M^{(a)}$ converges in law to a proper random variable. Furrer \cite{Fur99} and Resnick and Samorodnitsky \cite{RS00} proved similar results assuming that the random walk $\{S_n\}$ belongs to the domain of attraction of a spectrally positive stable law and using functional limit theorems. It is worth mentioning that Furrer has computed the corresponding limit distribution explicitly.

The {\em main purpose} of the present note is to determine the asymptotic behaviour of $M^{(a)}$ for any random walk from the domain of attraction of a stable law.

A random walk $\{S_n\}$ is said to belong to the domain of attraction of a stable law with index $\alpha\in(0,2]$, if there exist sequences $b_n$ and $c_n$ such that
$$
\frac{S_n-b_n}{c_n}\to\xi\text{ weakly},
$$
where $\xi$ has the corresponding stable distribution.

It is known that the random walk $\{S_n\}$ belongs to the domain of attraction of a stable law with index $\alpha\in(0,2]$ iff the function
$$
V(x):=\mathbf{E}(X_1^2,|X_1|\leq x),\ x>0
$$
is regularly varying at infinity with index $2-\alpha$. The latter implies that we can choose
$$
c_n:=\inf\Bigl\{u>0:\frac{V(u)}{u^2}\leq\frac{1}{n}\Bigr\},\ n\geq1.
$$
In this case the sequence $\{c_n\}$ is regularly varying with index $1/\alpha$ and, furthermore,
\begin{equation}
\label{eq:cn}
\frac{V(c_n)}{c_n^2} \sim \frac{1}{n}, n \to \infty.
\end{equation}
In this note we consider the case when $\{S_n\}$ belongs to the domain of attraction of a stable law with index $\alpha\in(1,2]$ and $\mathbf{E}X_1=0$. It is known that under these assumptions we can choose $b_n\equiv 0$. Hence, we have the convergence
$$
\frac{S_n}{c_n}\to\xi\text{ weakly}.
$$

Let $\{\xi_t,t\geq0\}$ denote a stable L{\'e}vy process, where $\xi_1$ is equal in law to $\xi$.
We also need to define the supremum of the process $\{\xi_t - t\}_{t \ge0}$ via
$$
M^* = \sup_{t \ge 0} \{\xi_t-t\},
$$
which is a proper random variable as $\xi_t-t$ drifts to $-\infty$ as $t \to \infty$.

For any $\alpha\in(1,2]$ one can choose a positive integer-valued function $n(a)$ such that
\begin{equation}
\label{eq:na}
a n(a) \sim c_{n(a)}\text{ as }a\to0.
\end{equation}
It follows from the regular variation of $\{c_n\}$ that $n(a)$ is regularly varying at zero with index $-\alpha/(\alpha-1)$.
This, in its turn, implies that $c_{n(a)}$ is regularly varying with index $-1/(\alpha-1)$.
\begin{theorem} 
Suppose that $\{S_n\}$ belongs to the domain of attraction of a stable law with index $\alpha\in(1,2]$.
Then, as $a\to 0$,
\begin{equation} \label{eq:result}
\frac{M^{(a)}}{c_{n(a)}} \to M^*\text{ in law.}
\end{equation}
\end{theorem}

The theorem implies that $M^{(a)}$ grows as a regularly varying at zero function with index $-1/(\alpha-1)$. The limit distribution (that of $M^*$) is only known explicitly in two particular cases. If $\{\xi_t\}$ has no positive jumps, then one has an exponential distribution. If $\{\xi_t\}$ has no negative jumps and $\alpha<2$, then $M^*$ has a Mittag-Leffler distribution (see, e.g., \cite{Fur99}). In the other cases the explicit form of the distribution is unknown, however, one can easily find its tail asymptotics $\mathbf P(M^* > x) \sim C x^{1-\alpha}$ as $x \to \infty$.

We will provide two different proofs of the theorem, with the use of both methods mentioned above. However, we would like to start by giving a brief description and comparison of the two methods.

The first, analytical, method is based on the analysis of the Laplace transform of the normalised maximum of a random walk. It was proposed by Kingman \cite{Kin61, Kin62, Kin65}. The idea is to use a corollary of the Wiener-Hopf factorisation:
$$
\mathbf E e^{-\mu \frac{M^{(a)}}{c_{n(a)}}}=\exp\left\{-\sum_{k=1}^\infty \frac{1}{k}
\mathbf E\left(1 - e^{-\mu \frac{S_k^{(a)}}{c_{n(a)}}}; S_k^{(a)} > 0\right)\right\},\ \mu\geq0.
$$
In \cite{Kin65} heuristic arguments are given stating that the main contribution to the infinite series in the exponent on the RHS is due to the values of $k$ of order $n(a)$. This, however, was not proved formally. Instead, the author represented the exponent in the form of an integral along the imaginary axis and gave a proof of the statement by solving a Wiener-Hopf boundary-value problem. The same method was used later by Boxma and Cohen \cite{BC99} and Cohen \cite{Cohen02} in the case of an infinite variance.

Our proof justifies the heuristic arguments of Kingman. The main difficulty consists in showing that the contribution to the infinite series on the RHS of the latter equality from values of $k \gg n(a)$ is negligible. This is shown in our proof to follow from
\begin{equation}
\label{cool1}
\lim_{T \to \infty} \sum\limits_{k \ge Tn(a)} \frac{1}{k} \mathbf P\left(S_k^{(a)} > 0\right) = 0
\end{equation}
uniformly in $a>0$. Thus, the derivation of (\ref{cool1}) will be the crucial step in our proof.

We now turn to the second method used in the literature, namely the method based on functional limit theorems. Consider the maximum of $\{S_n^{(a)}\}$ on a finite-time interval
$$
M_T^{(a)} = \sup_{k \le T {n(a)}} S_k^{(a)}.
$$
It easily follows from the functional limit theorems for asymptotically stable random walks that
$$
\lim_{a \to 0} \mathbf P\left(\frac{M_T^{(a)}}{c_{n(a)}} \ge x\right) = \mathbf P(M_T^* \ge x)
$$
for any $x \ge 0$, where
$$
M_T^* = \sup_{t \le T} \{\xi_t-t\}.
$$
The latter convergence implies that
$$
\lim_{T \to \infty} \lim_{a \to 0} \mathbf P\left(\frac{M_T^{(a)}}{c_{n(a)}} \ge x\right) = \mathbf P(M^* \ge x).
$$
However, in order to prove the theorem, one needs to show that
$$
\lim_{a \to 0} \lim_{T \to \infty} \mathbf P\left(\frac{M_T^{(a)}}{c_{n(a)}} \ge x\right) = \mathbf P(M^* \ge x).
$$
Therefore, it remains to justify the interchange of limits. It is easy to see that a sufficient condition could be written as
\begin{equation}
\label{cool2}
\lim_{T \to \infty} \mathbf P\left(\sup\limits_{k \ge {n(a)}T} S_k^{(a)} \ge 0\right) = 0
\end{equation}
uniformly in $a>0$. This was shown by Prokhorov \cite{Prokh63} in the case of a finite variance with the application of the classical Kolmogorov inequality. Later, Asmussen \cite[page 289]{Asm03} proved the same result using the fact that the sequence $\{S_n/n\}$ is a backward martingale. This fact was also utilised by Furrer \cite{Fur99} and Resnick and Samorodnitsky \cite{RS00} in the case of an infinite variance.

We are now going to state and to prove a generalisation of the Kolmogorov inequality which allows one to overcome the technical difficulties in both approaches described above, namely, to prove (\ref{cool1}) and (\ref{cool2}). A formal proof of the theorem using both approaches is provided after the proof of the following lemma, which is an easy consequence of Pruitt's bound, see \cite{Pruitt81}.

\begin{lemma}
There exists $C>0$ such that the inequality
\begin{equation}
\label{Chebyshev}
\mathbf{P}\Bigl(\max_{k\leq n}S_k\geq x\Bigr)\leq C\frac{nV(x)}{x^2}
\end{equation}
holds for all $x>0$.
\end{lemma}
\begin{proof}
According to inequality (1.2) in \cite{Pruitt81},
\begin{equation}
\label{eq:L1}
\mathbf{P}\Bigl(\max_{k\leq n}S_k\geq x\Bigr)\leq Cn\Bigl(\mathbf{P}(|X_1|>x)+
x^{-1}|\mathbf{E}(X_1;|X_1|\leq x)|+x^{-2}V(x)\Bigr)
\end{equation}
(here and in what follows $C$ denotes a generic positive and finite constant).

It is easy to see that
\begin{align*}
\mathbf{P}(|X_1|>x)&=\sum_{j=0}^\infty\mathbf{P}(|X_1|\in(2^jx,2^{j+1}x])\leq
\sum_{j=0}^\infty\frac{V(2^{j+1}x)}{2^{2j}x^2}\\
&\leq\frac{V(x)}{x^2}4C(\gamma)\sum_{j=1}^\infty 2^{-(\alpha-\gamma)j},
\end{align*}
where in the last step we used the inequality
\begin{equation}
\label{eq:L2a}
\frac{V(y)}{V(x)}\leq C(\gamma)\Bigl(\frac{y}{x}\Bigr)^{2-\alpha+\gamma},\quad y\geq x,
\end{equation}
which follows from the Karamata representation, see \cite[Theorem 1.2]{Sen76} (recall that
$V(x)$ is regularly varying with index $2-\alpha$). Choosing $\gamma<\alpha$, we get
\begin{equation}
\label{eq:L2}
\mathbf{P}(|X_1|>x)\leq C\frac{V(x)}{x^2}.
\end{equation}

In a way similar to that used in obtaining (\ref{eq:L2}), we can get the bound
\begin{equation}
\label{eq:L3}
|\mathbf{E}(X_1;|X_1|\leq x)|=|\mathbf{E}(X_1;|X_1|> x)|\leq C\frac{V(x)}{x}.
\end{equation}
Combining (\ref{eq:L1}), (\ref{eq:L2}) and (\ref{eq:L3}), we get inequality~(\ref{Chebyshev}).
The proof is thus complete.
\end{proof}
\begin{proof}[Proof of the theorem via Wiener-Hopf]
The Wiener-Hopf factorisation implies that
\begin{equation*}
\mathbf E e^{-\lambda M^{(a)}} =
\exp\left\{\sum_{k=1}^\infty \frac{1}{k} \mathbf E\left(e^{-\lambda S_k^{(a)}} - 1; S_k^{(a)} > 0\right)\right\}
\end{equation*}
for any $\lambda>0$, see, e.g. \cite[Proposition 19.2]{Sp01}. For the purposes of our proof, we set $\lambda = \mu/c_{n(a)}$ where $\mu>0$.
In these terms the latter equality takes the form
\begin{equation}
\label{WH_RW}
\mathbf E e^{-\mu \frac{M^{(a)}}{c_{n(a)}}}=\exp\left\{-\sum_{k=1}^\infty \frac{1}{k}
\mathbf E\left(1 - e^{-\mu \frac{S_k^{(a)}}{c_{n(a)}}}; S_k^{(a)} > 0\right)\right\}.
\end{equation}
We aim at showing that a pointwise, in $\mu$, limit of this Laplace transform is equal to the Laplace transform of $M^*$.

Fix $\varepsilon\in(0,1)$, $T>1$ and divide the sum in the exponent in (\ref{WH_RW}) into
three parts:
$$
\sum_{k=1}^\infty = \sum_{k=1}^{\varepsilon n(a)} + \sum_{k=\varepsilon n(a)}^{Tn(a)} + \sum_{k=Tn(a)}^{\infty}
=: \Sigma_1 + \Sigma_2 + \Sigma_3.
$$
We will now analyse these three summands separately.

It follows from the inequalities $0\leq 1-e^{-t}\leq t$ for any positive $t$ that
\begin{equation*}
0\leq \mathbf E\left(1 - e^{-\mu \frac{S_k^{(a)}}{c_{n(a)}}}; S_k^{(a)} > 0\right) \le
\frac{\mu}{c_{n(a)}}\mathbf E\left(S_k^{(a)};S_k^{(a)} > 0\right).
\end{equation*}
Note that $\mathbf{E}\left(S_k^{(a)};S_k^{(a)} > 0\right)=\mathbf{E}\left(S_k-ka;S_k>ka\right)\leq\mathbf E(S_k;S_k>0)$.
Since $S_n$ is asymptotically stable with $\alpha>1$,
$$
\lim_{n\to\infty}c_n^{-1}\mathbf{E}(S_n;S_n>0)=\mathbf{E}(\xi_1;\xi_1>0).
$$
This implies that $\mathbf{E}(S_k;S_k>0)\leq Cc_k$ for all $k\geq1$.
As a result we have the bound
$$
0\leq \mathbf E\left(1 - e^{-\mu \frac{S_k^{(a)}}{c_{n(a)}}}; S_k^{(a)} > 0\right) \le
C\mu\frac{c_k}{c_{n(a)}}.
$$
Hence,
\begin{equation}
\label{eq:T1}
0\leq \Sigma_1 \le \frac{C\mu}{c_{n(a)}} \sum_{k=1}^{\varepsilon n(a)} \frac{c_k}{k}.
\end{equation}
Since the sequence $\{c_k\}$ is regularly varying with index $1/\alpha\in(0,1)$,
$$
\sum\limits_1^n\dfrac{c_k}{k}\sim\dfrac{1}{\alpha}c_n\quad\text{ as }n\to\infty.
$$
Consequently,
\begin{equation}
\label{eq:T1a}
\sum_{k=1}^{\varepsilon n(a)} \frac{c_k}{k}\leq
C c_{\varepsilon n(a)}\leq C\varepsilon^{1/\alpha}c_{n(a)}.
\end{equation}
Estimates (\ref{eq:T1}) and (\ref{eq:T1a}) imply that
\begin{equation}
\label{eq:T1b}
0\leq \Sigma_1 \le C\varepsilon^{1/\alpha}\mu.
\end{equation}

In order to bound $\Sigma_3$ we note that
$$
0\leq \mathbf E\left(1 - e^{-\mu \frac{S_k^{(a)}}{c_{n(a)}}}; S_k^{(a)} > 0\right)\leq \mathbf{P}(S_k^{(a)}>0)=\mathbf{P}(S_k>ka).
$$
Using the lemma, we obtain
\begin{equation*}
0\leq \Sigma_3 \le \sum_{k=Tn(a)}^\infty \frac{1}{k} \mathbf P(S_k > ka)
\le C \sum_{k=Tn(a)}^\infty \frac{V(ka)}{(ka)^2}
\le C \frac{1}{a} \int_{Tan(a)}^\infty \frac{V(x)}{x^2}dx.
\end{equation*}
Recalling that $V(x)$ is regularly varying with index $2-\alpha$, we continue with
\begin{equation}
\label{eq:T2}
\Sigma_3\leq C\frac{V(Tan(a))}{Ta^2n(a)}
\leq C T^{1-\alpha} \frac{V(an(a))}{a^2n(a)}\leq C T^{1-\alpha}.
\end{equation}
In the last step we used the relations
\begin{equation}
\label{eq:T2b}
\frac{V(an(a))}{a^2n(a)}\sim\frac{V(c_{n(a)})}{(c_{n(a)})^2}n(a)\sim1,\ a\to0,
\end{equation}
which follow from (\ref{eq:cn}).

It now remains to analyse $\Sigma_2$. Using the assumption that $S_k/c_k$ converges in law to $\xi_1$, we get
\begin{align}
\label{eq:T2a}
\nonumber
\frac{S_k^{(a)}}{c_{n(a)}} &= \frac{S_k}{c_{n(a)}}-\frac{ka}{c_{n(a)}}\\
&=\frac{S_k}{c_k}\frac{c_k}{c_{n(a)}}-\frac{ka}{c_{n(a)}}\to v^{1/\alpha}\xi_1-v
\end{align}
in distribution as $a\to0$ and $k/n(a)\to v\in(0,\infty)$.
In the last step we used (\ref{eq:na}) and the regular variation of $\{c_k\}$.
It follows from the scaling property of the stable process $\{\xi_t\}$ that $v^{1/\alpha}\xi_1$
and $\xi_v$ are equal in law. From this relation and (\ref{eq:T2a}) we conclude that
$$
\mathbf E\left(1 - e^{- \mu \frac{S_k^{(a)}}{c_{n(a)}}}; S_k^{(a)} > 0\right)\to
\mathbf{E}\left(1-e^{-\mu(\xi_v-v)};\xi_v-v>0\right)
$$
as $a\to0$ and $k/n(a)\to v\in(0,\infty)$.
Using now the dominated convergence we see that, as $a\to0$,
\begin{align}
\label{eq:T3}
\nonumber
\Sigma_2 = \sum_{k = \varepsilon n(a)}^{Tn(a)}
\frac{n(a)}{k} \mathbf E\left(1 - e^{- \mu \frac{S_k^{(a)}}{c_{n(a)}}}; S_k^{(a)} > 0\right)\frac{1}{n(a)}\to\hspace{2cm}\\
\int_{\varepsilon}^{T} \frac{1}{v}\mathbf E \left(1 - e^{- \mu(\xi_v-v)} ; \xi_v-v > 0\right) dv.
\end{align}
It is clear that the process $\xi_t-t$ drifts to $-\infty$. Then, using \cite[Theorem 48.1]{Sato}, we obtain
\begin{align}
\label{eq:T3a}
\nonumber
\lim_{\varepsilon\to0,T\to\infty}\int_{\varepsilon}^{T} \frac{1}{v}\mathbf E \left(1 - e^{- \mu(\xi_v-v)} ; \xi_v-v > 0\right) dv\hspace{2cm}\\
=\int_{0}^{\infty} \frac{1}{v}\mathbf E \left(1 - e^{- \mu(\xi_v-v)} ; \xi_v-v > 0\right) dv<\infty.
\end{align}
Hence, combining (\ref{WH_RW}), (\ref{eq:T1b}), (\ref{eq:T2}) and (\ref{eq:T3}) and letting $\varepsilon\to0$, $T\to\infty$, we get
$$
\mathbf E e^{-\mu \frac{M^{(a)}}{c_{n(a)}}}\to
\exp\Bigl\{-\int_{0}^{\infty} \frac{1}{v}\mathbf E \left(1 - e^{- \mu(\xi_v-v)} ; \xi_v-v > 0\right) dv\Bigr\}
\quad\text{as }a\to0.
$$
The latter expression is known to be the Laplace transform of $M^*$, see again \cite[Theorem 48.1]{Sato}.
This completes the proof.
\end{proof}
\begin{proof}[Proof of the theorem via functional limit theorems]
Our proof in this case is closely related to the description given in the discussion after the statement of the theorem. However, for the sake of mathematical rigour, we provide here a complete proof.

According to the functional limit theorem for asymptotically stable random walks,
$$
\Bigl\{\frac{S_{[nt]}}{c_n}; t\in[0;T]\Bigr\}\to\{\xi_t;t\in[0;T]\}
$$
in distribution. Therefore, as $a\downarrow0$,
$$
\Bigl\{\frac{S_{[n(a)t]}-[an(a)t]}{c_{n(a)}}; t\in[0;T]\Bigr\}\to\{\xi_t-t;t\in[0;T]\}.
$$
This convergence implies that
$$
\lim_{a\downarrow0}\mathbf{P}\Bigl(\max_{k\leq n(a)T}S_k^{(a)}\geq xc_{n(a)}\Bigr)=
\mathbf{P}\Bigl(\sup_{t\leq T}(\xi_t-t)\geq x\Bigr),\ x\geq0.
$$
Since $\sup\limits_{t\leq T}(\xi_t-t)$ converges to $M^*$, The theorem will be proved, if
we show that
\begin{equation}
\label{eq:T4}
\lim_{T\to\infty}\mathbf{P}\Bigl(\max_{k\geq n(a)T}S_k^{(a)}\geq 0\Bigr)=0
\end{equation}
uniformly in $a>0$.

Note that
\begin{align*}
\Bigl\{\max_{k\geq n(a)T}S_k^{(a)}\geq 0\Bigr\}&=
\bigcup_{j=0}^\infty\Bigl\{\max_{k\in[2^jn(a)T,2^{j+1}n(a)T)}(S_k-ka)\geq 0\Bigr\}\\
&\subset\bigcup_{j=0}^\infty\Bigl\{\max_{k\leq2^{j+1}n(a)T}S_k\geq 2^jan(a)T\Bigr\}
\end{align*}
From this relation and the lemma we obtain
\begin{align*}
\mathbf{P}\Bigl(\max_{k\geq n(a)T}S_k^{(a)}\geq 0\Bigr)&\leq
\sum_{j=0}^\infty\mathbf{P}\Bigl(\max_{k\leq2^{j+1}n(a)T}S_k\geq 2^jan(a)T\Bigr)\\
&\leq C\sum_{j=0}^\infty\frac{2^{j+1}n(a)TV(2^jan(a)T)}{(2^jan(a)T)^2}\\
&= \frac{2C}{a}\sum_{j=0}^\infty\frac{V(2^jan(a)T)}{2^jan(a)T}.
\end{align*}
With the use of (\ref{eq:L2a}) with some $\gamma<\alpha-1$ the latter expression can be estimated by
$$
C\frac{V(an(a)T)}{a^2n(a)T}.
$$
The regular variation of $V(x)$ and (\ref{eq:T2b}) yield the bound
$$
\frac{V(an(a)T)}{a^2n(a)T}\leq CT^{1-\alpha},
$$
which implies (\ref{eq:T4}), completing our proof.
\end{proof}

At the end we would like to make a few remarks. First of all we
note that our assumption on the family of random walks can be
weakened. We assumed that $S_k^{(a)} = S_k -ka$ for the
transparency of all proofs. However, all our arguments remain
valid if we assume, for instance, that
$$
S_k^{(a)} = S_k -ka + \sum\limits_{i=1}^k Y_i^{(a)},
$$
where $\{Y_i^{(a)}\}$ is a sequence of i.i.d. random variables such that $\mathbf E Y_1^{(a)} = 0$ and $\mathbf E |Y_1^{(a)}|^{\gamma} < \infty$ uniformly in $a$ for some $\gamma > \alpha$. Indeed, one can easily verify that, uniformly in $a$,
$$
\frac{S_n^{(a)}+na}{c_n}\to\xi\quad\text{weakly}
$$
and, furthermore,
$$
\mathbf{E}\left(\left(S_1^{(a)}+a\right)^2;\left|S_1^{(a)}+a\right|\leq x\right)\leq C V(x)
$$
with the latter bound allowing us to apply the lemma to the random walk $\{S_n^{(a)}+na\}$.

Moreover, in the case of a finite variance, for our arguments to remain valid, it is sufficient that a Lindeberg-type condition holds, i.e.
$$
S_k^{(a)} = \sum\limits_{i=1}^k X_i^{(a)},
$$
where $\{X_i^{(a)}\}$ is a sequence of i.i.d. random variables such that $\mathbf E X_1^{(a)} = -a$, $\lim\limits_{a\to0}Var X_1^{(a)}=\sigma^2 \in (0, \infty)$ and
\begin{equation}
\label{Lin}
\lim\limits_{a\to0}\mathbf E\left(\left(X_1^{(a)}\right)^2; |X_1^{(a)}| > \frac{\varepsilon}{a}\right)=0 \quad \text{for all } \varepsilon>0.
\end{equation}
The latter condition seems to be optimal. On the one hand, it is necessary, according to the Lindeberg theorem, for the normal approximation of $aS_{1/a^2}^{(a)}$. On the other hand, Sakhanenko \cite{Sakh05} gave an example of a family $S^{(a)}$ such that (\ref{Lin}) fails and $M^{(a)}=0$ for all $a$.

\vspace{12pt}

{\em Acknowledgements.} This work was carried out during the visits of the first author to Technische Universit{\"a}t M{\"u}nchen and Ludwig-Maximilians-Universit{\"a}t M{\"u}nchen and the visit of the second author to EURANDOM. Both authors are grateful to the above-mentioned institutions for their hospitality. The authors are also thankful to Ron Doney for drawing their attention to paper \cite{Pruitt81} which helped in simplifying the proof of the lemma.

\end{document}